\documentclass[10pt]{amsart}

\usepackage{amssymb,amsthm,amsmath}
\usepackage[numbers,sort&compress]{natbib}
\usepackage{color}
\usepackage{graphicx}
\usepackage{tikz}
\usepackage{amssymb,amsthm,amsmath}
\usepackage[numbers,sort&compress]{natbib}
\usepackage{color}
\usepackage{graphicx}
\usepackage{tikz}

\usepackage{comment}

\title[ ]{Inhomogeneous Diophantine approximation in the coprime setting}
\author{ Svetlana Jitomirskaya}
\address[ Svetlana Jitomirskaya]{ Department of Mathematics, University of California, Irvine, California 92697-3875, USA}
\email{szhitomi@math.uci.edu}

\author{Wencai Liu}
\address[Wencai Liu]{Department of Mathematics, University of California, Irvine, California 92697-3875, USA} \email{liuwencai1226@gmail.com}
\address{Current address: Department of Mathematics, Texas A\&M University, College Station, TX 77843-3368, USA}

\newcommand{\phisymbol}{\xi}
\newcommand{\R}{\mathbb{R}}
\newcommand{\Z}{\mathbb{Z}}
\newcommand{\Q}{\mathbb{Q}}
\newcommand{\T}{\mathbb{T}}

\theoremstyle{plain}
\newtheorem{theorem}{Theorem}[section]
\newtheorem{corollary}[theorem]{Corollary}
\newtheorem{lemma}[theorem]{Lemma}

\theoremstyle{definition}

\newtheorem{remark}[theorem]{Remark}

\begin{document}
\renewcommand{\epsilon}{\phisymbol}

 \newcommand{\N}{\mathbb{N}}

\begin{abstract}

 Given $n\in\N$ and $x,\gamma\in\R$, let
   \begin{equation*}
    ||\gamma-nx||^\prime=\min\{|\gamma-nx+m|:m\in\Z,  \gcd (n,m)=1\},
   \end{equation*}
Two conjectures in the coprime inhomogeneous Diophantine approximation state that
for any irrational number $\alpha$ and almost every $\gamma\in \R$,
\begin{equation*}
    \liminf_{n\to \infty}n||\gamma -n\alpha||^{\prime}=0
\end{equation*}
and that there exists $C>0$, such that for all  $\alpha\in \R\backslash \Q$ and $\gamma\in [0,1)$ ,
 \begin{equation*}
    \liminf_{n\to \infty}n||\gamma -n\alpha||^{\prime} < C.
 \end{equation*}
 We prove the first conjecture and disprove the second one.

\end{abstract}
\maketitle

 \section{Introduction}

Let $\psi:\N\to \R^+$ and $\alpha\in\R\backslash\Q, \gamma\neq 0.$
A classical Diophantine approximation  problem studies the existence of infinitely many pairs $(p,q)$ of integers such that
\begin{equation}\label{psi}
    |\gamma-q\alpha+p|\leq \psi(q).
\end{equation}

It is referred to as homogeneous if $\gamma=0$ and inhomogeneous if
$\gamma\neq 0$. See \cite{mem} for the discussion of known
results and references.

Questions of this type have applications, among other things, to
several areas of dynamical systems and to the spectral
theory of quasiperiodic Schr\"odinger operators
(e.g. \cite{aj,ajm,jl1,jl2}).
The inhomogeneous problem above can be understood in the metric sense: a.e. $\gamma$, and in the
uniform sense: all $\gamma$.

Coprime inhomogeneous approximation asks the same questions about
infinitely many coprime pairs $(p,q)$. This question has been linked to the density exponents of lattice
orbits in $\R^2$, in \cite{laurent2012inhomogeneous}.

For the classical uniform setting,
Minkowski Theorem guarantees that for any irrational $\alpha\in \R$ and $\gamma\notin \alpha\Z+\Z$, there are
infinitely many pairs $(p,q)$ of   integers such that
\begin{equation}\label{14}
    |\gamma-q\alpha+p|\leq \frac{1}{4|q|}.
\end{equation}

Grace \cite{grace1918note} showed that $1/4$ in \eqref{14} is sharp, and Khintchine \cite{cusick1994inhomogeneous}
showed that
\begin{equation*}
    \liminf_{|q|\to \infty}|q|\; ||q\alpha-\gamma||\leq \frac{1}{4}(1-4\lambda(\alpha)^2)^{\frac{1}{2}},
 \end{equation*}
where
\begin{equation*}
    \lambda( \alpha)= \liminf_{|q|\to \infty}|q|\;||q\alpha||,
  \end{equation*}
  and $||x||=\text{dist}(x,\Z)$.

Uniform inhomogeneous coprime approximation was studied by Chalk and
Erd\'{o}s who proved  \cite{ce}  that for any  irrational $\alpha\in\R$
and for any $\gamma$  there are infinitely many pairs of
coprime integers $(p,q)$ such that (\ref{psi}) holds with $\psi(q)=
(\frac {\log q}{\log\log q})^2\frac{1}{q}.$

 Laurent and Nogueira \cite{laurent2012inhomogeneous} conjectured that a result similar to
Minkowski's theorem holds also for the inhomogeneous coprime
approximation, namely that there exists $C>0$ such that
  for any  for any  irrational $\alpha\in\R$
and for any $\gamma$ there are infinitely many pairs of coprime integers $(p,q)$ with
 \begin{equation*}
    |\gamma-q\alpha+p|\leq\frac{C }{ {|q|}}.
 \end{equation*}

In other words the conjecture is that $(\frac {\log q}{\log\log q})^2$ in Chalk-Erd\H{o}s
can be replaced by $C.$ Such a result would clearly be optimal up to determining the optimal $C.$

The work of Chalk and Erd\H{o}s was forgotten by the community until
recently and the problem was studied in several
papers (e.g. \cite{Alan}), where results somewhat
weaker than in \cite{ce} were obtained by different methods.
The best positive result towards this conjecture remains the one in \cite{ce}.

Our first result in the present paper is to show that such $C$ does
not exist. This shows
 that the coprime requirement leads to fundamental differences in the quality of
 approximation for the inhomogeneous setting.
\begin{theorem}\label{thmexp}
For any constant $C$, there exists $(\alpha,\gamma)\in [0,1)^2$ with
$\alpha\in  \R\backslash \Q  $ and $ \gamma\notin \alpha\Z+\Z$ such that
the inequality
 \begin{equation*}
    |\gamma-q\alpha+p|\leq\frac{C }{ {q}}.
 \end{equation*}
only has    finitely  many coprime solutions  $(p,q)\in \N^2$.
\end{theorem}
\begin{remark}
Actually, for both $\alpha$ and $\gamma$ corresponding bad sets can be shown to be dense and uncountable, see Theorem \ref{thmexp2}.
\end{remark}

Given $n\in\N$ and $x\in\R$, define
   \begin{equation*}
    ||x-n\alpha||^\prime=\min\{|x-n\alpha-m|:m\in\Z, (n,m)=1\},
   \end{equation*}
where $(n,m)$ is the largest positive common divisor of $n$ and $m$.

Addressing the inhomogeneous coprime approximation problem from the metric point,
Laurent and Nogueira proved that  (\ref{psi}) has infinitely many
coprime solutions $(p,q)$ for almost every $(\alpha,\gamma)\in \R^2$
provided $\sum \psi(n)=\infty.$ In particular, there are infinitely
many coprime solutions for almost every $(\alpha,\gamma)\in \R^2$ for
$\psi(n)=c/n.$ Laurent and Nougeira \cite{laurent2012inhomogeneous} conjectured that the same is true
on each fiber for a fixed $\alpha$, and they proved that
\begin{equation}\label{gal}
    \liminf_{n\to \infty}|n||\gamma -n\alpha||^{\prime}\leq 2
\end{equation}

for $\alpha$ such that $\sum_{k\geq
  0}\frac{1}{\max(1,\log q_k)}=\infty,$ where $q_k$ are denominators
of continued fraction approximants of $\alpha,$ and a.e. $\gamma.$
Condition (\ref{gal}) is essential for the proof of \cite{laurent2012inhomogeneous} because it
requires an application of Gallaher's theorem.

Our second result in this paper is a proof of (a stronger version
of) the above conjecture for all irrational $\alpha.$

We prove
\begin{theorem}\label{mainthm}
For any irrational number $\alpha$,
\begin{equation*}
    \liminf_{n\to \infty}n||\gamma -n\alpha||^{\prime}=0
\end{equation*}
holds for almost every $\gamma\in \R$.
\end{theorem}

\begin{remark}
\begin{itemize}
\item Since $||\gamma -n\alpha||^{\prime}$ is 1-periodic with respect
  to $\alpha$, we  always assume $\alpha\in (0,1)$ in this paper.
\item It is known \cite{kim2007} that for almost every $\gamma$,
\begin{equation}\label{Jan19}
 \liminf_{n\to \infty}n||\gamma
  -n\alpha||=0.
\end{equation}
  However,  the exceptional set of $\gamma$  of \eqref{Jan19} has full Hausdorff measure \cite{Yann}.
     A necessary and sufficient
  condition on
  $\psi$ and $\alpha$ so that $ \liminf_{n\to \infty}\psi(n)||\gamma
  -n\alpha||=0$ holds for a.e. $\gamma$  was given in \cite{kim2015}.
\end{itemize}

\end{remark}
%
Except  the generalized Borel-Cantelli lemma, basic facts on the distribution of prime numbers and some basic ergodic arguments,
the present paper is self-contained.

The rest of the paper is organized as follows:
In \S 2, we obtain the asymptotics of coprime pairs. In \S 3, we  will give the proof of Theorem \ref{mainthm}.
In \S 4, we will give the proof of  Theorem \ref{thmexp}.

The following standard notations will be used.
 Let $ (n,m)$ be the largest   common divisor of $n$ and $m$, and
 $\{x\}=x-\lfloor x\rfloor$, the fractional part of $x$.
Denote  by $|A|$  the Lebesgue measure  of $A$ and by $\# S$ the number of elements in $S$. Let $\T=\R/\Z$.  Without loss of generality, we always assume $\alpha\in(0,1)$ is irrational.
\section{The asymptotics of coprime pairs }

For $n\in \N$, let $\pi(n)$ be the number of prime number less than $n$.
It is well known that the prime numbers satisfy the following asymptotics \cite{MR0003018}
\begin{equation}\label{disprime3}
    \pi(n)=\frac{ n}{\ln n}(1+O(\frac{1}{\ln n})).
\end{equation}
By the distribution of prime numbers, we also have the following well known results:
a weaker version of Mertens' second theorem,
\begin{equation}\label{disprime1}
   \sum_{ 2\leq p< n\atop { p \text{ is prime}}}\frac{1}{p}=\ln \ln n+O(1),
\end{equation}
and a weaker version of Rosser's theorem (see \cite{MR1576808}),
\begin{equation}\label{disprime2}
   \sum_{  { p \text{ is prime}}}\frac{1}{p\ln p}<\infty.
\end{equation}

For any $\alpha\in [0,1)\backslash \Q$, we denote its continued fraction expansion by
\begin{equation*}
\alpha=[a_1,a_2,\cdots,a_n,\cdots]=\frac{1}{a_1+\frac{1}{a_2+\frac{1}{a_3+\frac{1}{\cdots}}}}.
\end{equation*}
Let
\begin{equation*}
\frac{p_n}{q_n}=[a_1,a_2,\cdots,a_n]=\frac{1}{a_1+\frac{1}{a_2+\frac{1}{a_3+\frac{1}{\cdots+\frac{1}{a_n}}}}},
\end{equation*}
where $(p_n,q_n)=1$.

By the properties of continued fraction expansion (see \cite{einsiedler2013ergodic} for example), one has
\begin{equation}\label{GDC1}
    \min_{p\in\Z}|k\alpha-p|\geq |q_n\alpha-p_n|
\end{equation}
for any $1\leq k<q_{n+1}$, and
\begin{equation}\label{GDC2}
      \frac{1}{q_n+q_{n+1}}\leq |q_n\alpha-p_n| \leq\frac{1}{q_{n+1}}.
\end{equation}
Moreover,
\begin{equation}\label{GDC3}
 q_n\alpha-p_n=(-1)^{n}   |q_n\alpha-p_n|.
\end{equation}
 In the following Sections 2 and 3,  $C$ is a large absolute constant.
Let
\begin{equation*}
    \kappa=\prod_{p \text{ is prime }}  (1-\frac{1}{p^2}).
\end{equation*}
It is well known that
\begin{equation*}
    \kappa=\frac{6}{\pi ^2}.
\end{equation*}
Set
\begin{equation}\label{GJ}
    J_n=\{k\in [1,q_{n+1}-1]\cap\Z: (\lfloor k\alpha\rfloor,k)=1\},
\end{equation}
where $\lfloor x\rfloor$ is the largest integer less or equal than $x$.

For an interval $I\subset [0,1)$, let
\begin{equation*}
    \hat{J}_n=\{k:k\in J_n, \{k\alpha\}\in I\}.
\end{equation*}
The following Theorem is crucial in our proof.

\begin{theorem}\label{thm1}


There exists a  sequence $\{n_k\}$ (independent  of $I$) such that
 the following asymptotics  hold as $k\to \infty$,
\begin{equation}\label{GasymJ}
    \#\hat{J}_{n_k}= (\kappa|I|+o(1))q_{n_k+1}.
\end{equation}

\end{theorem}
\begin{proof}

Set $J_n(p)$,
\begin{equation*}
    J_n(p)=\{k\in [1,q_{n+1}-1]\cap\Z: p|(\lfloor k\alpha\rfloor,k)\}.
\end{equation*}
and
\begin{equation*}
    \hat{J}_n(p)=\{k\in [1,q_{n+1}-1]\cap\Z: p|(\lfloor k\alpha\rfloor,k),\{k\alpha\}\in I \}.
\end{equation*}
 {\bf Claim 1:}  $k\in \hat{J}_n(p)$ if and only if  there exists some $k_1\in\N$ such that
\begin{equation}\label{equ1}
    k=pk_1
\end{equation}
and
\begin{equation}\label{equI}
    \{k_1\alpha\}\in \frac{I}{p},
\end{equation}
where
\begin{equation*}
    \frac{I}{p}=\{\frac{x}{p}:x\in I\}.
\end{equation*}
See a proof of the Claim 1  at the end of this Section.


Fix  $ \epsilon>0$ (small enough). Let $\varepsilon >0$ be sufficiently small.
 Now we distinguish the cases $p\leq \frac{1}{\epsilon}$ and $p>\frac{1}{\epsilon}$.

By ergodic theorem, for large $n$ (depending on $ \epsilon,\varepsilon$) one has
\begin{equation}\label{equ3}
 |I|\frac{q_{n+1}}{p^2}- \varepsilon q_{n+1}\leq \#\{k_1: 1\leq k_1<\frac{q_{n+1}}{p},\{k_1\alpha\}\in \frac{I}{p}\}\leq |I|\frac{q_{n+1}}{p^2}+\varepsilon q_{n+1}
\end{equation}
for all $p\leq  \frac{1}{\epsilon}$.
By Claim 1, one has
\begin{equation*}
|I|\frac{q_{n+1}}{p^2}-\varepsilon q_{n+1}\leq    \#\hat{J}_n(p)\leq |I|\frac{q_{n+1}}{p^2}+\varepsilon q_{n+1}.
\end{equation*}
By the definition of $J_n(p)$, we have for any prime numbers $p_1,p_2,\cdots, p_s$,
\begin{equation*}
    \hat{J}_n(p_1)\cap \hat{J}_n(p_2)\cap\cdots\cap \hat{J}_n(p_s)=\hat{J}_n(p_1p_2\cdots p_s).
\end{equation*}
Thus by the inclusion-exclusion principle, we have
\begin{equation*}
|I| (1-\varepsilon- \prod_{p\leq \frac{1}{\epsilon} \atop p \text{ is prime}} (1- \frac{1}{p^2}))q_{n+1}\leq \#\bigcup_{p\leq  \frac{1}{\epsilon}\atop p \text{ is prime}} \hat{J}_n(p)\leq |I|(1+\varepsilon- \prod_{p\leq \frac{1}{\epsilon}\atop p \text{ is prime}} (1- \frac{1}{p^2}))q_{n+1}.
\end{equation*}
This implies (letting $\varepsilon$ go to zero) that as $n\to \infty$,
\begin{equation*}
\#\{k:1\leq k<q_{n+1}, \text{ there exists some prime number }  2\leq p\leq \frac{1}{\epsilon} \text{ such that } p|(\lfloor k\alpha\rfloor,k),\text { and } k\alpha\in I\}
\end{equation*}
\begin{equation}\label{equ4}
   = |I|q_{n+1}[1 + o(1) -\prod_{p\leq \frac{1}{\epsilon}\atop p \text{ is prime}} (1- \frac{1}{p^2})].
\end{equation}


Now we are in a position to study the case $p>\frac{1}{\epsilon}$.
We will prove that there exists a sequence $\{n_k\}$ such that
\begin{equation}\label{Glargep}
    \#\bigcup_{p >\frac{1}{\epsilon}\atop{p \text{ is prime }}} J_{n_k}(p)=\varphi(\epsilon) q_{n_k+1}
\end{equation}
as $k\to\infty$, where $\varphi(\epsilon)$ goes to $0 $ as $\epsilon\to 0.$

We will split all  primes $p$ into the cases  $\frac{q_{n+1}}{C}<p<q_{n+1}$, $C q_n\leq p\leq \frac{q_{n+1}}{C}$, $ \frac{q_n}{C}\leq p\leq  Cq_n$ and
$ \frac{1}{\epsilon}< p\leq  \frac{q_n}{C}$, where $C$ is a large constant.

{\bf Case 1: }$\frac{q_{n+1}}{C}<p<q_{n+1}$

 By \eqref{equ1}, one has
 \begin{equation}\label{equk1}
    k_1\leq \frac{q_{n+1}}{p}.
 \end{equation}
 This  leads to $k_1\leq C$ in the current case. By \eqref{disprime3},   we have
 \begin{eqnarray*}
   \#\bigcup_{\frac{q_{n+1}}{C}<p<q_{n+1}\atop{p \text{ is prime }}} J_n(p)&\leq& C \#\{p: \frac{q_{n+1}}{C}<p<q_{n+1} \text{ and } p \text{ is prime}\} \\
     &\leq&C \frac{ q_{n+1}}{\ln q_{n+1}} =o(1)q_{n+1}.
 \end{eqnarray*}

 {\bf Case 2: } $ \frac{q_n}{C}\leq p\leq  Cq_n$

 By \eqref{disprime3} and \eqref{equk1} again, one has
 \begin{eqnarray*}
   \#\bigcup_{ \frac{q_n}{C}\leq p\leq  C q_n\atop{p \text{ is prime }}} J_n(p)&\leq& C\frac{q_{n+1}}{q_n} \#\{p:  \frac{q_n}{C}\leq p\leq  Cq_n \text{ and } p \text{ is prime}\} \\
     &\leq&C \frac{ q_{n+1}}{q_n}\frac{q_n}{\ln q_n} =o(1)q_{n+1}.
 \end{eqnarray*}

 {\bf Case 3:} $C q_n\leq p\leq \frac{q_{n+1}}{C}$.

 If $q_{n+1}\leq Cq_n$, there is no such $p$. We are done.
   Thus, we assume
 \begin{equation}\label{qnqn1}
    q_{n+1}\geq Cq_n.
 \end{equation}

In this case, one  has
 \begin{equation}\label{pl}
    k_1=\ell q_n
 \end{equation}
 for some $\ell\in \N$. Indeed, suppose  $k_1=\ell q_n +j_k$ with $1\leq j_k<q_{n}$.
 By \eqref{GDC1}, \eqref{GDC2} and $\ell\leq \frac{q_{n+1}}{Cq_n^2}$, one has
 \begin{eqnarray*}
   \{k_1\alpha\} &=& \{\ell q_n\alpha +j_k\alpha\} \\
    &\geq& ||j_k\alpha||- \ell|| q_n\alpha||\\
   &\geq& \frac{1}{2q_n} -\frac{q_{n+1}}{C q_n^2} \frac{1}{q_{n+1}}\\
    &\geq& \frac{1}{4q_n}.
 \end{eqnarray*}
 This  contradicts  \eqref{equI}.  

If $n$ is odd, by \eqref{GDC2} and \eqref{GDC3}, we have
\begin{eqnarray*}
  \{k_1\alpha\} &=& 1-||k_1\alpha||\\
   &=& 1-||\ell q_n\alpha||\\
   &\geq& 1-\frac{q_{n+1}}{p q_n}\frac{1}{q_{n+1}}\\
    &\geq& \frac{1}{2},
\end{eqnarray*}
which is impossible since $\{k_1\alpha\}<\frac{1}{p}$. This means there is no such $k$ in the current case.

Assume $n$ is even.
Suppose $ k\in J_n(p) $ and $ k\notin J_n(p^\prime) $  for prime $p^{\prime}<p$.

{\bf Claim 2:}
\begin{equation*}
    k=\ell p q_n, \text{ and } p^{\prime}\nmid \ell.
\end{equation*}
See a proof at the end of this Section.

 Thus for any prime $p$ in this case, we have
 \begin{eqnarray*}
    \# (J_n(p)\backslash \bigcup_{Cq_n<p^{\prime}<p\atop {p^\prime \text{ is prime}}} J_n(p^{\prime})) &\leq& \#\{\ell:\ell\leq\frac{q_{n+1}}{pq_n},  \ell \text{ does not have any divisor } p^{\prime} \text{ with }C q_n<p^{\prime}<p \} \\
    &\leq& \frac{q_{n+1}}{pq_n}\prod_ {C q_n<p^{\prime}<p\atop {p^\prime \text{ is prime}}}(1-\frac{1}{p^{\prime}}) \\
    &\leq&   C\frac{q_{n+1}}{pq_n}\frac{\ln q_n}{\ln p}\\
     &\leq&   C\frac{ q_{n+1}\ln q_n}{q_n}\frac{1}{p\ln p},
 \end{eqnarray*}
where the third inequality holds by \eqref{disprime1}.


 Thus by \eqref{disprime2}, we have
 \begin{eqnarray*}
    \# \bigcup _{Cq_n<p<\frac{q_{n+1}}{C}\atop{p \text{ is prime}}}J_n(p) &\leq&  C\frac{ q_{n+1}\ln q_n }{q_n}  \sum_{Cq_n<p<\frac{q_{n+1}}{C}\atop{p \text{ is prime}}}\frac{1}{p\ln p} \\
     &=& o(1)q_{n+1}.
 \end{eqnarray*}

 {\bf Case 4: }$ \frac{1}{\epsilon}\leq p\leq  \frac{q_n}{C}$

%

 For each $k_1<\frac{q_{n+1}}{p}$, rewrite $k_1=\ell q_n+\ell_{k_1}$ with $1\leq \ell_{k_1}<q_n$.
 Then by \eqref{GDC2}, one has
 \begin{equation*}
    ||\ell q_n \alpha||=\ell ||q_n\alpha||<\frac{1}{p}.
 \end{equation*}

 In this case, we must have
 \begin{equation}\label{Gcase42}
    ||\ell_{k_1}\alpha||\leq \frac{C}{p}.
 \end{equation}
 Indeed, if $||\ell_{k_1}\alpha||\geq \frac{C}{p}$,  by \eqref{GDC2}, one has
 \begin{eqnarray*}
   \{k_1\alpha\} &=& \{\ell q_n \alpha+\ell_{k_1}\alpha\} \\
    &\geq& ||\ell_{k_1}\alpha|| -||\ell q_n \alpha||\\
    &\geq& ||\ell_{k_1}\alpha|| -\frac{1}{p}\\
     &\geq& \frac{C}{p}.
 \end{eqnarray*}
 This is impossible since $\{k_1\alpha\}<\frac{1}{p}$.

 Also, by \eqref{GDC1} and \eqref{GDC2}, one has
 \begin{equation}\label{Gcase41}
    \#\{j:1\leq j<q_{n},||j\alpha||<\frac{C}{p}\}\leq C\frac{q_n}{p}+1\leq C\frac{q_n}{p}.
 \end{equation}
 By \eqref{Gcase42} and \eqref{Gcase41},
   we have
 \begin{eqnarray}
  \# \bigcup _{\frac{1}{\epsilon}<p<\frac{q_{n}}{C}\atop{p \text{ is prime}}}J_n(p) &\leq& \sum_{\frac{1}{\epsilon}<p<\frac{q_{n}}{C}\atop{p \text{ is prime}}} C(\frac{q_{n+1}}{pq_n} +1)\frac{q_n}{p} \nonumber\\
  &\leq& \sum_{\frac{1}{\epsilon}<p<\frac{q_{n}}{C}\atop{p \text{ is prime}}}  C(\frac{q_{n+1}}{pq_n})\frac{q_n}{p}+C\sum_{\frac{1}{\epsilon}<p<\frac{q_{n}}{C}\atop{p \text{ is prime}}}\frac{q_{n}}{p}\nonumber \\
     &=&\varphi(\epsilon)q_{n+1}+C q_{n}\ln \ln q_{n},\label{last1}
 \end{eqnarray}
where  $\varphi(\epsilon)\to 0$ as $\epsilon\to 0$.

 Suppose  there exists an infinite    sequence $\{n_k\}$ such that
\begin{equation}\label{A}
    q_{n_k+1}\geq q_{n_k}(\ln \ln q_{n_k})^2.
\end{equation}
Then by \eqref{last1}, 
we have
\begin{equation*}
     \# \bigcup _{\frac{1}{\epsilon}<p<\frac{q_{n_k}}{C}\atop{p \text{ is prime}}}J_{n_k}(p) =\varphi(\epsilon)q_{n_k+1}.
\end{equation*}
Putting   the other cases together, this completes the proof of \eqref{Glargep}.

Otherwise for all large $s$, we have
\begin{equation}\label{Gnogrowth}
    q_{s+1}\leq q_{s}(\ln\ln q_{s})^2.
\end{equation}

For any $p$, let $s $ be the unique positive integer such that
\begin{equation}\label{ps}
    q_{s}\leq p<q_{s+1}.
\end{equation}
Suppose $k_1,k_1^\prime\in J_n(p)$. We must have
\begin{equation}\label{gqs}
    |k_1-k_1^\prime|\geq q_{s-2}.
\end{equation}
Otherwise, by \eqref{GDC1} and \eqref{GDC2}, one has
\begin{equation*}
    ||k_1\alpha-k_1^\prime\alpha||>\frac{1}{2q_{s-2}}.
\end{equation*}
This is impossible since $\{k_1\alpha\}<\frac{1}{p}$,$\{k_1^\prime\alpha\}<\frac{1}{p}$ and $p\geq q_{s}$.

By \eqref{Gnogrowth}, \eqref{ps} and \eqref{gqs},
for any $k_1,k_1^\prime\in J_n(p)$,  one has
\begin{equation*}
    |k_1-k_1^\prime|\geq p^{\frac{1}{2}}.
\end{equation*}
Thus
\begin{eqnarray*}
  \# \bigcup _{\frac{1}{\epsilon}<p<\frac{q_{n}}{C}\atop{p \text{ is prime}}}J_n(p) &\leq& \sum_{p>\frac{1}{\epsilon}\atop{p \text{ is prime}}} C\frac{q_{n+1}}{p p^{\frac{1}{2}}} \\
     &=& \varphi(\epsilon)q_{n+1}.
 \end{eqnarray*}
Putting   the other cases together, we finish the proof of \eqref{Glargep}.

Now the Theorem follows from \eqref{equ4} and \eqref{Glargep}  by letting $\epsilon \to 0$.

%
%
%

\end{proof}
{\bf Proof of   Claim 1}
\begin{proof}
Suppose $k\in \hat{J}_n(p)$. Then  there exist $k_1,q_1\in\N$ such that
\begin{equation*}
    k=pk_1
\end{equation*}
and
\begin{equation*}
    \lfloor k\alpha\rfloor=pq_1.
\end{equation*}

Using $ k_1\alpha=\lfloor k_1\alpha\rfloor+\{k_1\alpha\}$, one has
\begin{equation*}
     \lfloor k\alpha\rfloor= \lfloor pk_1\alpha\rfloor=p\lfloor k_1\alpha\rfloor+\lfloor p\{k_1\alpha\}\rfloor.
\end{equation*}
This implies
\begin{equation*}
   p |\lfloor p\{k_1\alpha\}\rfloor.
\end{equation*}
Noting  that $0\leq \lfloor p\{k_1\alpha\}\rfloor\leq p-1$, one has
\begin{equation}\label{equ2}
    \lfloor p\{k_1\alpha\}\rfloor=0.
\end{equation}
Thus
\begin{equation*}
    \{k_1\alpha\}<\frac{1}{p}.
\end{equation*}
Combining with the assumption that $\{k\alpha\}\in I$, one has
\begin{equation*}
   p\{k_1\alpha\}=\{k\alpha\}\in I.
\end{equation*}
This yields that
\begin{equation*}
    \{k_1\alpha\}\in \frac{I}{p}.
\end{equation*}
The proof of the other side   is similar. We omit the details.
\end{proof}

{\bf Proof of   Claim 2}
\begin{proof}
Otherwise, $k=p^{\prime}p m q_n$ for some $p^{\prime}<p$.
Since $n$ is even, by \eqref{GDC3},
one has
\begin{equation*}
    ||q_n\alpha||=\{q_n\alpha\}<\frac{1}{q_{n+1}}.
\end{equation*}
Thus
\begin{eqnarray*}
  \{ p m q_n\alpha\} &=& pm||q_n\alpha|| \\
  &<& \frac{q_{n+1}}{p^{\prime}q_n} \frac{1}{q_{n+1}}=\frac{1}{p^{\prime}q_n}.
\end{eqnarray*}
This implies
\begin{eqnarray*}
  \lfloor k\alpha\rfloor &=& \lfloor p^\prime p m q_n \alpha\rfloor \\
    &=& p^\prime \lfloor p m q_n \alpha\rfloor +\lfloor p^{\prime}\{p m q_n \alpha\}\rfloor\\
     &=& p^\prime \lfloor p m q_n \alpha\rfloor.
\end{eqnarray*}
 Thus
 \begin{equation*}
    p^{\prime}|(\lfloor k\alpha\rfloor,k).
 \end{equation*}
 We get a contradiction since $(\lfloor k\alpha\rfloor,k)\notin J(p^\prime)$.

\end{proof}
%
%
%
\begin{remark}\label{nodepend}
We should mention that the   sequence  $\{n_k\}$ in Theorem \ref{thm1} is either defined by \eqref{A} or is the entire sequence $n\in\N$ in case \eqref{Gnogrowth}. So it  does not depend on the interval $I$.
\end{remark}
\section{Proof of Theorem \ref{mainthm}}
We present the general form   of the Borel-Cantelli Lemma first, which is the key technique  in  this part of the argument. See \cite{dodson2007diophantine,sprindzhuk1979metric}  for details.
\begin{lemma}
\label{thmleb}
Let $E_k$, $k=1,2,\cdots$, be a sequence of Lebesgue measurable sets in
$ [0,1] $
and suppose that
 \begin{equation}\label{eqleb}
\sum_{k=1}^\infty |E_k| =\infty.
 \end{equation}
 Then the Lebesgue measure of $E:=\limsup_{N\to\infty}
 E_N:=\bigcap_{N=1}^\infty \bigcup_{k=N}^\infty E_k$ satisfies
\begin{equation*}
 |E|  \ge  \limsup_{N\to\infty}\frac{(\sum_{k=1}^N |E_k|)^2}
 {\sum_{k=1}^N\sum_{l=1}^N |E_k\cap E_l|}.
\end{equation*}

\end{lemma}
  Lemma \ref{thmleb} immediately implies
\begin{corollary}\label{cor}
Suppose the  sets  $\{E_k\}$ are pairwise quasi-independent  with respect to constant $A>0$, that is
\begin{equation*}
    |E_k \cap E_l |\leq A|E_k||E_l|+C2^{-(k+l)},
\end{equation*}
for all $k\neq l$, and $\sum_{k=1}^\infty |E_k| =\infty$. Let $E:=\limsup_{N\to\infty}
 E_N:=\bigcap_{N=1}^\infty \bigcup_{k=N}^\infty E_k$.
 Then \begin{equation*}
 |E|  \ge \frac{1}{A}.
\end{equation*}
\end{corollary}
Define
\begin{equation}\label{GI}
    I_n=\bigcup_{k\in J_n}(\{k\alpha\}-\frac{\tau}{q_{n+1}},\{k\alpha\}+\frac{\tau}{q_{n+1}}),
\end{equation}
where   $0<\tau<\frac{1}{3}$, and $J_n$ is given by \eqref{GJ}.

By \eqref{GDC1} and \eqref{GDC2}, we have
 $I_n\subset [0,1]$ and
 $I_n$ contains exactly $\# J_n$ intervals.

Letting $I=[0,1)$ in Theorem \ref{thm1}, we have
\begin{lemma}\label{lefinity}
Let  sequence $\{n_k\}$ be given by Theorem \ref{thm1}. Then we have
\begin{equation*}
    \# {J}_{n_k}= (\kappa +o(1))q_{n_k+1}
\end{equation*}
and
\begin{equation*}
    |I_{n_k}|= 2\tau\kappa+o(1).
\end{equation*}

\end{lemma}
Now we want to show the quasi-independence of a subsequence  of  $\{I_{n_k}\}$.
\begin{theorem}\label{thm2}
Fixing  $n_{{l_1}} \in  \{{n_k}\}$, we have
\begin{equation}\label{GasymInm}
    |I_{n_{l_1}}\cap I_{n_{l_2}}|= |I_{n_{l_1}}|\;|I_{n_{l_2}}|+o(1),
\end{equation}
as   $n_{l_2}\in  \{{n_k}\}$ goes to infinity.

\end{theorem}
\begin{proof}
Recall  that $I_n$ contains $\# J_n$ intervals. Let $I_n=\bigcup I_n^i$, $i=1,2,\cdots,\#J_n$.
Fix one interval $I_{n_{l_1}}^i $ with $1\leq i\leq \#J_{n_{l_1}}$.
Now let $l_2 $ go to infinity.
By Theorem \ref{thm1}, one has
\begin{equation*}
   \# \{k:1\leq k<q_{n_{l_2}+1}:  \{k\alpha\}\in  I_{n_{l_1}}^i  ,(\lfloor k\alpha\rfloor,k)=1 \}=q_{n_{l_2}+1}\kappa|I_{n_{l_1}}^i|+o(1)q_{n_{l_2}+1}.
   \end{equation*}
   By \eqref{GI}, one has
   \begin{equation*}
    |I_{n_{l_2}}\cap I_{n_{l_1}}^i|=2\tau \kappa |I_{n_{l_1}}^i|+o(1).
   \end{equation*}

   By Lemma \ref{lefinity}, we have
   \begin{equation*}
    |I_{n_{l_2}}\cap I_{n_{l_1}}^i| =|I_{n_{l_2}}|\; |I_{n_{l_1}}^i|+o(1).
   \end{equation*}
   Summing up all the $i\in \#J_{n_{l_1}}$, we  obtain the Theorem.
\end{proof}
{\bf Proof of Theorem \ref{mainthm}}
\begin{proof}
We give the proof of $\gamma\in[0,1)$ first.
Applying Theorem \ref{thm2}, there exists
a sequence $\{n_{k_l}\}$  such that
\begin{equation*}
    |I_{n_{k_i}}\cap I_{n_{k_j}}|=|I_{n_{k_i}}|\;|I_{n_{k_j}}|+\frac{O(1)}{2^{i+j}}.
\end{equation*}
Letting $E_l=I_{n_{k_l}}$, by Lemma \ref{lefinity}, one has
\begin{equation*}
    |E_l|= 2\tau\kappa+o(1).
\end{equation*}
Applying Corollary \ref{cor} with $A=1$, we get $|\limsup E_l|=1$.

By the definition of $I_n$, we have
for any $\gamma\in I_n$,   there exists some $1\leq k<q_{n+1}$ such that
$(\lfloor k\alpha\rfloor,  k)=1$  and
\begin{equation*}
    |\gamma-\{k\alpha\}|= |\gamma- k\alpha+\lfloor k\alpha\rfloor|\leq   \frac{\tau}{q_{n+1}}.
\end{equation*}
This implies   for any $n_{k_l}$, there exists some $1\leq j<q_{n_{k_l}+1}$ such that
\begin{equation*}
   ||\gamma- j\alpha||^{\prime}\leq  \frac{\tau}{q_{n_k+1}}\leq \frac{\tau}{j}.
\end{equation*}
 Since  $\tau$ is  arbitrary,  we have
for almost every $\gamma\in[0,1)$,
\begin{equation*}
  \liminf_{k\to \infty} k ||\gamma- k\alpha||^{\prime}=0.
\end{equation*}
Let  us now consider $\gamma\in[m,m+1)$ for some $m\in \Z$.
In this case, we only need to modify the definition of $I_n$ in \eqref{GI} as
\begin{equation*}
    I_n^m=\bigcup_{k\in J_n}(m+\{k\alpha\}-\frac{\tau}{q_{n+1}},m+\{k\alpha\}+\frac{\tau}{q_{n+1}}).
\end{equation*}
By the same proof as for $\gamma\in[0,1)$, we have for almost every $\gamma\in[m,m+1)$,
\begin{equation*}
  \liminf_{k\to \infty} k ||\gamma- k\alpha||^{\prime}=0.
\end{equation*}
This completes the proof.
\end{proof}
%
%
\section{Proof of Theorem \ref{thmexp}}
In this section, we will prove the following theorem, which is  a finer  version of Theorem \ref{thmexp}.
\begin{theorem}\label{thmexp1}
For any positive constant $M$, there exist $\alpha\in \R\backslash\Q$ and an uncountable set $\Omega\subset \R$(depending on $\alpha$) such that
for all $\gamma\in \Omega$,
the inequality
 \begin{equation}\label{Mar182019}
    |\gamma-q\alpha+p|\leq\frac{M }{ {q}}.
 \end{equation}
only has finitely many coprime solutions  $(p,q)\in \N^2$.
\end{theorem}
We need some preparations first. In this section, all the large  constants $C,C_1$ and $C_2$  only depend on $M$.
In the following arguments, we assume $C_2,C_1,C\in \N$ and
$$C_2>>C_1>>C.$$
Let $p^1=2,p^2,p^3,\cdots,p^n$ be the successive prime numbers with some $n= (2C_1+1)^2$ and let $$P=p^1p^2\cdots p^n.$$

Define $a_k=\hat{l}_kP$ for $\hat{l}_k\in \N$, $k=1,2,\cdots$. In the following construction, we need that $\hat{l}_{k}>C_2$.
Let $\alpha=[a_1,a_2,\cdots,a_k,\cdots]$ and $\frac{p_k}{q_k}=[a_1,a_2,\cdots,a_k]$.
Then $p_0=0, p_1=1,q_0=1, q_1=a_1$ and
\begin{equation}\label{Gpk}
     p_{k}=a_kp_{k-1}+p_{k-2},
\end{equation}
and
\begin{equation}\label{Gqk}
    q_{k}=a_kq_{k-1}+q_{k-2}.
\end{equation}
Thus, we have
\begin{equation}\label{Gmqk}
    q_k\equiv 1\mod P,
\end{equation}
and
\begin{equation}\label{Gmpkeven}
    p_{2k}\equiv 0\mod P,
\end{equation}
and
\begin{equation}\label{Gmpkodd}
    p_{2k+1}\equiv 1\mod P.
\end{equation}
Since $\hat{l}_{k}>C_2$, we have
\begin{equation}\label{Ggrowthqk}
    q_{k+1}\geq C_2q_k.
\end{equation}

 We assign each pair $(t,j) $, $|t|\leq C_1, |j|\leq C_1$ a different prime number $p^{t,j}$. We can randomly choose $p^{t,j}$ so that
$p^{t,j}$ is a permutation  of  a subset of prime numbers $p^1,p^2,\cdots, p^n$ with $n=(2C_1+1)^2$.  We also assume $\alpha$ is given by \eqref{Gpk}-\eqref{Ggrowthqk}.

The plan is to construct a sequence   $ \{b_k\}$ such that  for all $k$,
$ b_{k}\equiv t\mod p^{t,j}$, $ \lfloor b_{k}\alpha\rfloor\equiv j\mod p^{t,j},$
for all  $|t|\leq C_1,|j| \leq C_1$.
This will be done by induction.
We then construct nested  intervals $\{I_k\}\subset [0,1)$ centered at $b_k\alpha\mod \Z$  and $\lim |I_k|=0$.
We will show that for $ \gamma=\cap I_k$ there are only finitely many
coprime solutions to \eqref{Mar182019} .  Here is the sketch of the argument.

Suppose   \eqref{Mar182019}  has infinitely many coprime solutions. We
will show (Theorem \ref{keythmdi}) that solutions $(p,q)$  must have the structure $p=\lfloor q\alpha\rfloor$ and $q=b_{k}+d_k q_k+r_kq_{k-1}$ with $  |d_k|\leq C$ and $|r_k|\leq C$ for some $k$.  By  \eqref{Gmqk}-\eqref{Gmpkodd}, the remainders of $d_k q_k+r_kq_{k-1}, \lfloor (d_k q_k+r_kq_{k-1})\alpha\rfloor
\mod p^{t,j}$ for all $|t|,|j|\leq C_1$ are bounded by  $(2C+1)^2$. It will imply that for some $(t_0,j_0)\in [-C_1,C_1]\times [-C_1,C_1]$, both
 $\lfloor q\alpha\rfloor \mod p^{t_0,j_0}$ and $ q \mod p^{t_0,j_0} $ are zero.
This is a contradiction.

To start with the construction of $b_k$, clearly, we can find $ b_1\equiv t\mod p^{t,j}$ by the Chinese
Remainder Theorem. Simultaneously achieving $ \lfloor
b_1\alpha\rfloor\equiv j\mod p^{t,j}$ requires
$b_1\alpha/p^{t,j} \mod \Z$ belonging to a certain interval
of length $1/p^{t,j}.$  In fact, in order to proceed with inductive
construction of $b_k$, we will need a little more: that we can
guarantee $b_1\alpha/p^{t,j}\mod \mathbb{Z}$ in slightly shrunk
intervals. The following lemma is a preparation for that.

%
\begin{lemma}\label{lema4.2}
Suppose  $\hat{p}^1,\hat{p}^2,\cdots \hat{p}^k$ are distinct    prime numbers.  Let $\hat{P}=\hat{p}^1\hat{p}^2\cdots \hat{p}^k$.
Then there exists a small $\delta>0$ and a large constant  $\bar{L}>0$
(both depending on $ \hat{p}^1,\hat{p}^2,\cdots \hat{p}^k$)  such that
for any $\alpha$ with $a_1\geq \bar{L}$, any given box  $I=I_1\times
I_2\times \cdots \times I_k\subset \T^k$ with $|I_i|\geq
\frac{1}{\hat{p}^i}-\delta$ for $i=1,2,\cdots,k$,  any $L\geq a_1+1$ and  any $L_0,$
there exists
some $j\in\{L_0,L_0+1,\cdots,L_0+L\}$ such that
\begin{equation*}
    (j\frac{\hat{P}}{\hat{p}^1}\alpha,j\frac{\hat{P}}{\hat{p}^2}\alpha,\cdots, j\frac{\hat{P}}{\hat{p}^k}\alpha)\in I \mod \Z^k.
\end{equation*}
\end{lemma}
\begin{proof}
Let us consider the map $\psi_k:\R\to \T^k$,
\begin{equation}\label{Gpsit}
    \psi_k(t)=(t,\frac{\hat{p}^1}{\hat{p}^2}t,\cdots,\frac{\hat{p}^1}{\hat{p}^k}t)\mod \Z^k.
\end{equation}
See Fig.1. Since $\psi_k(\hat{p}^2\hat{p}^3\cdots \hat{p}^k)=0\mod \Z^k$, identifying $\T^k$ with $[0,1)^k$, we have that there exists some $N=N(\hat{p}^1,\hat{p}^2,\cdots, \hat{p}^k)$   such that the image of $\psi_k\subset [0,1)^k$ consists  of at most
$N$ segments.
\begin{center}
\begin{tikzpicture}[xscale=1.5,yscale=0.9]

\draw [-](0,0)--(0,5);
\draw [-](0,0)--(3,0);
\draw [-](3,5)--(3,0);
\draw [-](0,5)--(3,5);

\draw[-](0,0)--(3,3);
\draw[-](1,0)--(3,2);
\draw[-](2,0)--(3,1);

\draw[-](0,1)--(3,4);
\draw[-](0,2)--(3,5);
\draw[-](0,3)--(2,5);
\draw[-](0,4)--(1,5);

\draw[-](1,1)--(1,2);
\draw[-](1,1)--(2,1);
\draw[-](1,2)--(2,2);
\draw[-](2,2)--(2,1);

\draw[-](1.1,1.1)--(1.1,1.9);
\draw[-](1.1,1.1)--(1.9,1.1);
\draw[-](1.1,1.9)--(1.9,1.9);
\draw[-](1.9,1.9)--(1.9,1.1);

\draw[-](1.7,3.0)--(1.7,4.0);
\draw[-](1.7,3.0)--(2.7,3.0);
\draw[-](1.7,4.0)--(2.7,4.0);
\draw[-](2.7,4.0)--(2.7,3.0);
\
\draw[-](1.8,3.1)--(1.8,3.9);
\draw[-](1.8,3.1)--(2.6,3.1);
\draw[-](1.8,3.9)--(2.6,3.9);
\draw[-](2.6,3.9)--(2.6,3.1);
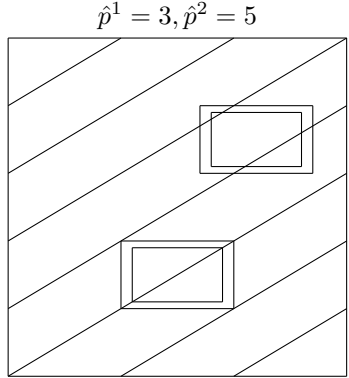
\node [below] at (1.5,0){Fig.1: two prime numbers};
\node [above] at (1.5,5){ $\hat{p}^1=3,\hat{p}^2=5$};
\end{tikzpicture}
\end{center}
{\bf Claim 3:} For any closed box  $\hat{I}^k=\hat{I}_1\times \hat{I}_2\times \cdots \times \hat{I}_k$ with $|\hat{I}_i|= \frac{1}{\hat{p}^i}$ for $i=1,2,\cdots,k$,  and for any $t_1\in \hat{I}_1$, there exists some
$t\equiv t_1\mod \Z$ such that  $\psi_k(t)\in \hat{I}^k \mod \Z$.

We will prove Claim 3  by induction. For $k=1$, it is trivial. Suppose it holds for $k$. Thus for any $t_1\in \hat{I}_1$, there exists some $t_2$ such that
$t_2\equiv t_1\mod \Z$ and $\psi_k(t)\in \hat{I}^k$. Since all the $\hat{p}^i$ are prime numbers, there exists  $q$ such that
$q \hat{p}^1\hat{p}^2\cdots\hat{p}^k\equiv 1\mod \hat{p}^{k+1}$. Then there exists some $0\leq j\leq \hat{p}^{k+1}-1$ such that
\begin{equation*}
  jq \hat{p}^1\hat{p}^2\cdots\hat{p}^k \frac{1}{\hat{p}^{k+1}} +\frac{\hat{p}^1}{\hat{p}^{k+1}} t_2\in \hat{I}_{k+1} \mod \Z,
\end{equation*}
since $|\hat{I}_{k+1} |=\frac{1}{\hat{p}^{k+1}}$ and $\hat{I}_{k+1} $ is closed.
Let $t_3= jq \hat{p}^2\cdots\hat{p}^k+t_2$.
Then $\frac{\hat{p}^1}{\hat{p}^{k+1}}t_3\in \hat{I}_{k+1}$ and $\psi_k(t_3)\in \hat{I}_k$ since $t_3\equiv t_2\mod \Z$.
Thus $\psi_{k+1}(t_3)\in \hat{I}^{k+1}$.\qed

Thus by Claim 3  and the fact that
 the image of $\psi_k\subset [0,1)^k$ consists of at most $N$ segements, for  any closed interval  $\hat{I}^k=\hat{I}_1\times \hat{I}_2\times \cdots \times \hat{I}_k$ with $|\hat{I}_i|= \frac{1}{\hat{p}^i}$ for $i=1,2,\cdots,k$,
 there exists some $\hat{\hat{I}}_1\subset [0,1) $  with $|\hat{\hat{I}}_1|\geq \frac{1}{2N \hat{p}^1}$  and $ 0\leq \hat{j}_0<\hat{p}^2\hat{p}^3\cdots\hat{p}^k$ such that   $\psi_k(t) \mod \Z^k\in \hat{I}^k $ for all $t\in \hat{\hat{I}}_1 +\hat{j}_0$. We mention that we use the fact that  $\psi_k$ is a map with period
 $\hat{p}^2\hat{p}^3\cdots\hat{p}^k$.

Let $\bar{L}= 3N \hat{P}$ and take $\alpha$ with $a_1\geq L.$ By the continued fraction expansion
  the set $\{ \alpha,2 \alpha,\cdots,   L \alpha\}$ is
$\frac{1}{3N \hat{P}}$ dense on the torus $\T$ if $L\geq a_1+1$. Let $0<\delta<<\frac{1}{12N \hat{P}}$. Now we will show  that  $L$ and $\delta$ satisfy the requirements of Lemma \ref{lema4.2}.

Indeed, suppose box $I=I_1\times I_2\times \cdots \times I_k$ has  $|I_i|\geq \frac{1}{\hat{p}^i}-\delta$ for $i=1,2,\cdots,k$.
Then, since the slopes of components of $\psi_k$ are bounded from
below by $\min\{\frac{\hat{p}^1}{\hat{p}^k}\}\geq \frac{\hat{p}^1}{\hat{P}}$,
there exists some $\tilde{I}_1\subset \hat{\hat{I}}_1$
with $ |\tilde I_1|\geq
\frac{1}{3N\hat{p}^1}$
such that  for any $m\in \Z$, $\psi_k(t) \in I \mod \Z^k $ for all $t\in m \hat{p}^2\hat{p}^3\cdots\hat{p}^k+j_0+\tilde{I}_1 $. We mention that we use again the fact that  $\psi_k$ is a map with period
 $\hat{p}^2\hat{p}^3\cdots\hat{p}^k$.
By the fact that  the set $\{ \alpha,2 \alpha,\cdots,  L \alpha\}$  is $\frac{1}{3N \hat{P}}$ dense on torus $\T$, we  have that  there exists some $j\in\{L_0,L_0+1,\cdots,L_0+L\}$  and $m_0\in \Z$ such that
\begin{equation*}
    j\alpha\in m_0+\frac{j_0}{\hat{p}^2\hat{p}^3\cdots\hat{p}^k}+\frac{\tilde{I}_1 }{\hat{p}^2\hat{p}^3\cdots\hat{p}^k}.
\end{equation*}
This implies
\begin{equation*}
    t_j=j\frac{\hat{P}}{\hat{p}^1}\alpha\in m_0 \hat{p}^2\hat{p}^3\cdots\hat{p}^k+j_0+\tilde{I}_1,
\end{equation*}
and then
\begin{equation*}
   \psi_k(t_j) =(j\frac{\hat{P}}{\hat{p}^1}\alpha,j\frac{\hat{P}}{\hat{p}^2}\alpha,\cdots, j\frac{\hat{P}}{\hat{p}^k}\alpha)\in I \mod \Z^k.
\end{equation*}

\end{proof}
\begin{lemma}\label{keyle1}
Let $p$ be a prime number. Suppose
\begin{equation*}
    b \equiv t\mod p.
\end{equation*}
Then
$\lfloor b\alpha +\gamma_j\rfloor\equiv j \mod p$ iff
\begin{equation*}
      \frac{(b-t)}{p} \alpha  \in [\frac{j-t\alpha-\gamma_j}{p},\frac{j+1-t\alpha-\gamma_j}{p})\mod \Z.
\end{equation*}
\end{lemma}
\begin{proof}
Let $b=kp+t$.
Suppose $\lfloor b\alpha+\gamma_j\rfloor\equiv j \mod p$.
Using $ k \alpha=\lfloor k\alpha\rfloor+\{k\alpha\}$, one has
\begin{equation*}
     \lfloor b\alpha +\gamma_j\rfloor=  p \lfloor  k\alpha\rfloor+\lfloor p \{k\alpha\}+t\alpha+\gamma_j\rfloor \equiv j \mod p
\end{equation*}
This implies
\begin{equation*}
     k \alpha  \in [\frac{j-t\alpha-\gamma_j}{p},\frac{j+1-t\alpha-\gamma_j}{p})\mod \Z.
\end{equation*}
The proof of the  other  side is similar. We omit the details.
\end{proof}
In the following, we always assume $a_1\geq \bar{L}$.
\begin{lemma}\label{keyle2}
There exist  a small $\delta>0$ (independent of $\alpha$) and
 $b_1\in \N$ such that for all $|t|,|j|\leq C_1$,
\begin{equation}\label{Gbk}
    b_1\equiv t\mod p^{t,j}
\end{equation}
and
\begin{equation}\label{Gbkinitial}
     \frac{(b_1-t)}{p^{t,j}} \alpha \in (\frac{j+\delta-t\alpha}{p^{t,j}},\frac{j+1-\delta-t\alpha}{p^{t,j}})\mod \Z.
\end{equation}
\end{lemma}
\begin{proof}
By  the Chinese remainder theorem, there exists $b$ such that
\begin{equation*}
    b_0\equiv t\mod p^{t,j}
\end{equation*}
for all $|t|,|j|\leq C_1$.

If $b_1=b_0+lP$, we also have
\begin{equation*}
    b_1\equiv t\mod p^{t,j}
\end{equation*}
for all $|t|,|j|\leq C_1$.

Suppose $\delta>0$ is small enough (only depends on $p^{t,j}$).
We only need to choose proper $l\in \N$  such that for all $t$ and $j$,
\begin{equation}\label{Gkeyle21}
    \frac{(b_1-t)}{p^{t,j}} \alpha \in (\frac{j+\delta-t\alpha}{p^{t,j}},\frac{j+1-\delta-t\alpha}{p^{t,j}})\mod \Z.
\end{equation}
Let $P^{t,j}=\frac{P}{p^{t,j}}$.
Thus \eqref{Gkeyle21} is equivalent to
\begin{equation*}
     P^{t,j} l \alpha \in (\frac{j+\delta-t\alpha}{p^{t,j}},\frac{j+1-\delta-t\alpha}{p^{t,j}})\mod \Z.
\end{equation*}
The existence of such  $l$ is guaranteed by Lemma \ref{lema4.2}.
\end{proof}

Now we will  construct nested  intervals $\{I_k\}\subset [0,1)$ such that $I_{k+1}\subset I_k $ and $\lim |I_k|=0$.
Here is the detail. Below, $l_k$ is always in $\N$.

Let $b_1$ be given by Lemma \ref{keyle2}.
Using the fact that $ \hat{l}_k>C_2$, one has $ \frac{q_2}{2}>b_1$. Thus we can
define
\begin{equation*}
b_{2}=b_1+l_1P q_1
\end{equation*}
such that
\begin{equation*}
    |b_{2}-\frac{q_{2}}{2}|\leq C P q_1.
\end{equation*}
Inductively, for $k\geq 2$,
define
\begin{equation}\label{keyle}
b_{k+1}=b_k+l_kP q_k
\end{equation}
such that
\begin{equation}\label{star}
    |b_{k+1}-\frac{q_{k+1}}{2}|\leq C P q_k.
\end{equation}
Let us
define
\begin{equation}\label{Ggammadis1}
    I_k=\{\gamma:\gamma\in[0,1), ||\gamma-b_k\alpha||\leq \frac{1}{q_k}\},
\end{equation}
where $b_k$ is given by   \eqref{keyle}.

Let
\begin{equation}\label{Ggammadis}
    \gamma=\cap I_k.
\end{equation}
\begin{remark}\label{Redense}
By modifying $b_1$ and  $l_k$ in construction of $ b_{k+1}$,  we can get a dense and uncountable set of $\gamma$.
\end{remark}
\begin{lemma}\label{keyle3}
Under the construction of \eqref{keyle}, we have that  for all $t,j$,
\begin{equation}\label{Gbk1t}
    b_{k+1}\equiv t\mod p^{t,j}
\end{equation}
and
\begin{equation}\label{Gbkinitialk}
     \frac{(b_{k+1}-t)}{p^{t,j}} \alpha \in (\frac{j+\frac{\delta}{2}-t\alpha}{p^{t,j}},\frac{j+1-\frac{\delta}{2}-t\alpha}{p^{t,j}})\mod \Z.
\end{equation}
\end{lemma}
\begin{proof}
The proof of \eqref{Gbk1t} follows from the definition of \eqref{Gbk} and \eqref{keyle}.

We will  give the proof of \eqref{Gbkinitialk} by induction. The base case holds by
 Lemma \ref{keyle2}. Suppose
\begin{equation}\label{Gbkinitialk1}
     \frac{(b_{k}-t)}{p^{t,j}} \alpha \in (\frac{j+\frac{\delta}{2}+2^{-(k+2)}\delta-t\alpha}{p^{t,j}},\frac{j+1-\frac{\delta}{2}-2^{-(k+2)}\delta-t\alpha}{p^{t,j}})\mod \Z.
\end{equation}
By   \eqref{GDC1} and \eqref{GDC2}, one has
\begin{eqnarray*}
   || \frac{ b_{k+1}-b_k}{p^{t,j}}\alpha|| &=&|| l_k P^{t,j}q_k\alpha||\\
   &\leq& \frac{q_{k+1}}{2q_k}\frac{1}{q_{k+1}} \\
    &\leq& \frac{1}{2q_k}\leq\frac{1}{ C_2^k},
\end{eqnarray*}
where the last inequality holds by \eqref{Ggrowthqk}.
By \eqref{Gbkinitialk1},  for appropriately large $C_1$, we have
\begin{equation*}
  \frac{(b_{k+1}-t)}{p^{t,j}} \alpha \in (\frac{j+\frac{\delta}{2}+2^{-k-3}\delta-t\alpha}{p^{t,j}},\frac{j+1-\frac{\delta}{2}-2^{-k-3}\delta-t\alpha}{p^{t,j}})\mod \Z.
\end{equation*}
Then by induction, we finish the proof.
\end{proof}
Thus in order to prove Theorem \ref{thmexp1}, we only need to show that for  the $\gamma$ given by \eqref{Ggammadis},
the inequality
 \begin{equation*}
    |\gamma-q\alpha+p|\leq\frac{M }{ {q}}.
 \end{equation*}
only has finitely coprime solutions  $(p,q)\in \N^2$.

Before, we give the proof, we need another theorem.
\begin{theorem}\label{keythmdi}
Suppose $\gamma$ is given by \eqref{Ggammadis}. Suppose $b_{k}\leq q <b_{k+1}$ and
\begin{equation*}
    ||\gamma-q\alpha||\leq\frac{M }{ {q}}.
\end{equation*}
Then $q$ must have the following form
\begin{itemize}
  \item  Case I: $q=b_{k+1}-a q_k$ with $0\leq a\leq C$.
  \item  Case II: $q=b_{k}+a q_k+bq_{k-1}$ with $0\leq a\leq C$ and $|b|\leq C$.
\end{itemize}

\end{theorem}
\begin{proof}
Suppose $b_k\leq q<b_{k+1}$ and
\begin{equation}\label{Gqtwo}
    ||\gamma-q\alpha||\leq\frac{M }{ {q}}.
\end{equation}
{\bf Case I:} $q>Cq_k$.

In this case, we   claim that $q=b_{k+1}-aq_k$ for some $a\geq0$.
Otherwise  $q=b_{k+1}-aq_k+l$ for some $1\leq l<q_{k}$.
Thus
\begin{eqnarray*}
  ||\gamma-q\alpha|| &=& ||\gamma-b_{k+1}\alpha+aq_{k}\alpha-l\alpha|| \\
   &\geq &  ||l\alpha|| -||a q_k\alpha||-||\gamma-b_{k+1}\alpha||\\
   &\geq& \frac{1}{q_k+q_{k-1}} -\frac{b_{k+1}}{q_k}\frac{1}{q_{k+1}}-\frac{1}{q_{k+1}}\\
    &\geq& \frac{1}{4q_{k}},
\end{eqnarray*}
where the second inequality holds by \eqref{GDC1}, \eqref{GDC2} and \eqref{Ggammadis1}, and the third inequality holds by the fact $q_{k+1}>C_2q_k$.
This contradicts \eqref{Gqtwo} since  $q\geq Cq_k$.

Now we are in a position to show $0\leq a\leq C$. Suppose $a>C$.
By \eqref{star} and $q=b_{k+1}-aq_k$, one has
\begin{equation}\label{a}
  a\leq \frac{q_{k+1}}{2q_k}+CP.
\end{equation}
By \eqref{Gqtwo}, we have
\begin{equation*}
   ||\gamma-q\alpha||\leq \frac{M}{b_{k+1}-aq_k}
\end{equation*}
and also, using (\ref{a}),
\begin{eqnarray*}
    ||\gamma-q\alpha|| &\geq& ||a q_k \alpha||-||\gamma-b_{k+1}\alpha|| \\
  &\geq& \frac{a}{q_{k+1}+q_{k}}-\frac{1}{q_{k+1}}\\
  &\geq&\frac{a}{2q_{k+1}}.
\end{eqnarray*}
Thus we have
\begin{equation}\label{Gsolve}
 \frac{a}{2q_{k+1}} \leq  \frac{M}{b_{k+1}-aq_k}.
\end{equation}
Solving  quadratic inequality \eqref{Gsolve}, we have by \eqref{star},
\begin{equation}\label{casse1}
    a\geq \frac{b_{k+1}+\sqrt{b_{k+1}^2-8Mq_kq_{k+1}}}{2q_k}=\frac{b_{k+1}}{q_k} +O(1),
\end{equation}
or
\begin{equation}\label{casse2}
    a\leq \frac{b_{k+1}-\sqrt{b_{k+1}^2-8Mq_kq_{k+1}}}{2q_k}=O(1).
\end{equation}
Inequality \eqref{casse1} can not happen since $q=b_{k+1}-aq_k\geq C q_k$.    Inequality \eqref{casse2} does not hold since we assume $a>C$.  This implies Case I.

{\bf Case II:} $ b_k\leq q\leq Cq_k$.
Rewrite $q$ as $q=b_k+aq_k+bq_{k-1}+l$, where $|b q_{k-1}+l|\leq \frac{1}{2}q_k$ and $   |l|<q_{k-1}$.
Notice that $|b|\leq \frac{q_k}{2q_{k-1}}$.

We claim that $l=0$. Indeed, assume  $|l|>0$. Then
\begin{eqnarray*}
  ||\gamma-q\alpha|| &=& ||\gamma-b_k\alpha-aq_k\alpha-bq_{k-1}\alpha-l\alpha|| \\
   &\geq& ||l\alpha||- ||\gamma-b_k\alpha||-||aq_k\alpha||-||bq_{k-1}\alpha||\\
   &\geq& \frac{1}{q_{k-1}+q_{k-2}}-\frac{1}{q_{k}}-\frac{a}{q_{k+1}}- \frac{q_k}{2q_{k-1}}\frac{1}{q_k}\\
   &\geq& \frac{1}{3q_{k-1}}
\end{eqnarray*}
where the second inequality holds by \eqref{GDC1},\eqref{GDC2} and \eqref{Ggammadis1}, and the third inequality holds by the fact $q_{k+1}>C_2q_k$.
This contradicts \eqref{Gqtwo} since  $q\geq b_k$.

In this case ($q\leq Cq_k$), it is immediate  that $0\leq a\leq C$. Thus we only need to prove
$|b|\leq C$.
Since $q=b_{k}+aq_k+bq_{k-1}$ and $q\geq b_k$, we have using also \eqref{star} that
\begin{equation*}
   ||\gamma-q\alpha||\leq \frac{M}{b_k}\leq   \frac{3M}{q_k}
\end{equation*}
and also have
\begin{eqnarray*}
    ||\gamma-q\alpha|| &\geq& ||b q_{k-1} \alpha||-||\gamma-b_k\alpha||-||aq_k \alpha|| \\
  &\geq& \frac{|b|}{q_{k-1}+q_{k}}-\frac{1}{q_{k+1}}-\frac{C}{q_k}
\end{eqnarray*}
Thus we have
\begin{equation*}
 |b|\leq    C.
\end{equation*}

\end{proof}

\begin{proof}[\bf Proof of Theorem \ref{thmexp1}]
Let $\gamma$ be  given by \eqref{Ggammadis}.
Suppose $b_k \leq q< b_{k+1}$ is such that $(p,q)$ are coprime and
\begin{equation*}
    |\gamma-q\alpha+p| \leq \frac{M}{q}.
\end{equation*}
Since $\gamma\neq 0$,
this implies $p=\lfloor q\alpha\rfloor$ and $||\gamma-q\alpha||\leq \frac{M}{q}$.
By Theorem \ref{keythmdi}, we have
$q=b_{k+1}-a q_k$ with $0\leq a\leq C$ or  $q=b_{k}+a q_k+bq_{k-1}$ with $0\leq a\leq C$ and $|b|\leq C$.

We will show that   $(q,\lfloor q\alpha\rfloor)$ is not coprime for all such $q$.

 Without loss of generality, assume $q=b_{k+1}-aq_k$ and $0\leq a\leq C$, the other part of the argument being similar.

 Suppose $l>0$ and $l<C$. By \eqref{Gmpkeven}, \eqref{Gmpkodd} and \eqref{GDC3}, we have for odd  $k$
\begin{equation}\label{Gbalpha1}
    \lfloor l q_k \alpha\rfloor =lp_k -1\equiv l-1\mod P
\end{equation}
and for even $k$
\begin{equation}\label{Gbalpha11}
    \lfloor l q_k \alpha\rfloor =lp_k \equiv 0\mod P.
\end{equation}
Suppose $l<0$ and $|l|<C$. Similarly,
we have for odd  $k$
\begin{equation}\label{Gbalpha1new}
    \lfloor l q_k \alpha\rfloor =lp_k \equiv l \mod P
\end{equation}
and for even $k$
\begin{equation}\label{Gbalpha11new}
    \lfloor l q_k \alpha\rfloor =lp_k -1\equiv -1\mod P.
\end{equation}
Let $\langle x \rangle$ be  the unique number in $[-1/2, 1/2)$
such that $x- \langle x\rangle$ is an integer.
Let $t=a$. Let  $0\leq -j<P$ be such that  $-j\equiv   -a q_k \alpha -\langle  -a q_k \alpha \rangle  \mod P$ and
  $\gamma_j= \langle  -a q_k \alpha \rangle$.

By \eqref{Gbalpha1}-\eqref{Gbalpha11new}, we have $0\leq |t|,|j|\leq C$.

By \eqref{GDC1} and \eqref{GDC2}, one has
\begin{equation}\label{Ggammaj}
|\gamma_j|=||  -a q_k \alpha||\leq \frac{C}{q_{k+1}}.
\end{equation}
By    \eqref{Gbk1t},
we have for all $|t|\leq C_1$,
\begin{equation*}
      b_{k+1} \equiv t \mod p^{t,j},
\end{equation*}
which implies for some $t \in[-C, C]$ (using $t=a$ and \eqref{Gmqk} )
\begin{equation}\label{Gnotcoprime1}
    b_{k+1}-aq_k\equiv 0 \mod p^{t,j}.
\end{equation}
Applying  Lemma \ref{keyle3} and \eqref{Ggammaj}, one has for large $k$,
\begin{equation*}
     \frac{(b_{k+1} -t)}{p^{t,j}} \alpha \in (\frac{j+\gamma_j-t\alpha}{p^{t,j}},\frac{j+1-\gamma_j-t\alpha}{p^{t,j}})\mod \Z.
\end{equation*}
By Lemma \ref{keyle1}, one has for all $|j|\leq C_1$
\begin{equation*}
    \lfloor b_{k+1}\alpha+\gamma_j\rfloor\equiv j\mod p^{t,j}.
\end{equation*}
This implies for some $j\in[-C,C]$,
\begin{equation}\label{Gnotcoprime2}
     \lfloor b_{k+1}\alpha-a q_k\alpha\rfloor\equiv  \lfloor b_{k+1}\alpha+\gamma_j\rfloor- j\equiv 0\mod p^{t,j}.
\end{equation}
Thus by \eqref{Gnotcoprime1} and \eqref{Gnotcoprime2},  we have that $(b_{k+1}-aq_k, \lfloor b_{k+1}-aq_k\rfloor)$ is not coprime.
This implies for such
$\gamma$ given by  \eqref{Ggammadis},
the inequality
 \begin{equation*}
    |\gamma-q\alpha+p|\leq\frac{M }{ {q}}.
 \end{equation*}
only has finitely many coprime solutions  $(p,q)\in \N^2$.
By Remark \ref{Redense},
this completes  the proof.

\end{proof}

Actually, we have proved the following more general result.
\begin{theorem}\label{thmexp2}
For any positive constant $M$, there exist large constants $\bar{C}_1$ and $\bar{C}_2$(depending on $M$) such that the following statement holds:
Let $P=p^1p^2\cdots p^n$, where $n= \bar{C}_1^3+1$. Let $\frac{p_k}{q_k}$ be the continued fraction expansion to $\alpha$.
Let
\begin{eqnarray}
    \Lambda &=& \{\alpha: \text{ there exists some }  S=\{a_1,a_2,\cdots, a_m\}\subset \N \cap[0,P]  \text { with } m\leq \bar{C}_1 \text{ such that} ,  \nonumber \\
   &&  \text{   eventually for all } k, \; q_k,p_k\in S  \mod P \text { and } q_{k+1}\geq  \bar{C}_2 P q_k\}.\nonumber
\end{eqnarray}
Then for any $\alpha\in \Lambda$, there exists a dense uncountable set $ \Omega (\alpha)\subset[0,1)$ such  that
for all $\gamma\in \Omega(\alpha)$,
inequality
 \begin{equation*}
    |\gamma-q\alpha+p|\leq\frac{M }{ {q}}.
 \end{equation*}
only has finitely many coprime solutions  $(p,q)\in \N^2$.
\end{theorem}
\begin{remark}
$\Lambda$ is a dense uncountable set.
\end{remark}
 \section*{Acknowledgments}
  We thank  Alan Haynes   for comments on the previous version of the
  manuscript. W.L. was supported by the AMS-Simons Travel
Grant 2016-2018.  This research was
 supported by NSF DMS-1401204, DMS-1901462 and NSF DMS-1700314.

\footnotesize
 \bibliographystyle{abbrv} 

\end{document}